\documentclass[11pt,twoside]{article}

\usepackage{mathrsfs,amsfonts,amsmath,amssymb}
\usepackage{latexsym}
\usepackage{comment} 
\newtheorem{satz}{Theorem}
\newtheorem{proposition}[satz]{Proposition}
\newtheorem{theorem}[satz]{Theorem}
\newtheorem{lemma}[satz]{Lemma}

\newtheorem{corollary}[satz]{Corollary}

\def\L{\mathcal{L}}

\def\Z{\mathbb {Z}}
\def\F{\mathbb {F}}
\def\E{\mathsf{E}}

\def\I{{\cal I}}
\def\a{\alpha}

\def\C{\mathbb{C}}
\def\P{{\cal P}}

\def\d{\delta}

\def\({\big (}
\def\){\big )}

\def\G{\Gamma}

\def\le{\leqslant}
\def\ge{\geqslant}
\def\_phi{\varphi}
\def\eps{\varepsilon}

\def\Gr{{\mathbf G}}

\def\ov{\overline}

\def\t{\tilde}

\def\Cf{{\mathcal C}}
\def\la{\lambda}
\def\D{\Delta}

\def\T{\mathsf{T}}

\def\C{\mathbb{C}}

\def\SL{{\rm SL}}

\def\Aff{{\rm Aff}}

\def\F{\mathbb {F}}
\def\R{{\mathbb R}}

\def\LL{{\mathcal L}}

\author{Shkredov I.D.}
\title{On a paucity result in Incidence Geometry  
\footnote{This work is supported by the Russian Science Foundation under grant 19--11--00001.}
}
\date{}
\begin{document}
	\maketitle

%\begin{document}

\begin{center}
	Annotation.
\end{center}

{\it \small
    We obtain some asymptotic formulae (with power savings in their error terms) for the number of quadruples in the  Cartesian product of an arbitrary  set $A \subset \R$ and for the number of quintuplets in $A\times A$ for  any subset $A$ of the prime field $\F_p$. Also, we obtain some applications of our results to incidence problems in $\F_p$.
}
\\

\section{Introduction}
\label{sec:introduction}

Incidence Geometry (see, e.g., \cite[section 8]{TV}) deals with the incidences among different geometrical objects such as points, lines, curves, surfaces etc. A typical problem of this area is to estimate the quantity 
\begin{equation}\label{def:I(P,L)} 
    \mathcal{I} (\mathcal{P},\mathcal{L}) :=|\{(p,l)\in \mathcal{P} \times \mathcal{L}:\,p\in l\}| \,,
\end{equation}
where the set of points $\mathcal{P}$ and the set of lines $\mathcal{L}$ belong to $\F\times \F$ with  $\F$ be a field, say. 
In our paper the set of points $\mathcal{P}$ will be  the Cartesian product $A\times B$ for some sets $A, B \subseteq \F$. This particular choice of $\P$ is very important for the applications see, e.g., \cite{collinear}, \cite{RS_SL2}, \cite{s_asymptotic}, \cite{s_He}, \cite{Sol_dZ}, \cite{TV} because, basically, Cartesian products are naturally  connected with arithmetic. 
In this note  we study collinear tuples in $A\times B$. 
Namely, for any $k \ge 3$
%and arbitrary sets $A,B \subseteq \F$ 
we define $\Cf_k(A,B)$ to be the number of collinear $k$--tuples in $A\times B$. Let $\Cf_k(A) = \Cf_k(A,A)$. 
A consequence of the famous Szemer\'edi--Trotter Theorem \cite{ST} gives us that for any $A\subset \R$ one has   
\begin{equation}\label{f:Q_3_R}
    \Cf_3 (A) \ll |A|^4 \log |A| 
    %\,.
\end{equation}
(a short  proof of this result contains in \cite{Sol_dZ}). 
Actually, for $\P = A\times A$ it is easy to see that bound \eqref{f:Q_3_R} is {\it equivalent} to the Szemer\'edi--Trotter Theorem (up to logarithms). 
In our paper \cite{collinear} we have obtained an asymptotic formula for the number of collinear quadruples in the case of the prime field $\F=\F_p$  
\begin{equation}\label{f:Q_4}
    \Cf_4 (A) = \frac{|A|^8}{p^2} + O(|A|^5 \log |A|) \,.
\end{equation}
    The proof is based on Stevens--de Zeeuw's Theorem \cite{SdZ}.
    Both bounds \eqref{f:Q_3_R}, \eqref{f:Q_4} are known to be tight up to some powers of logarithms as the case $A=\{1,\dots, n\}$ shows, namely, $\Cf_3 (A) \gg |A|^4 \log^c |A|$, $c>0$ for such $A$ (see discussion in \cite[Pages 603, 633]{collinear}).
    The problem of finding sharp  asymptotic formula for the number of  collinear triples in $A\times A$, $A\subseteq \F_p$ is a well--known important open question and at the moment the best result is (see \cite{collinear}, \cite{s_asymptotic}) 
\begin{equation}\label{f:Q_3}
    \Cf_3 (A) = \frac{|A|^6}{p} + O(\min\{|A|^{9/2}, p^{1/2} |A|^{7/2} \}) \,.
\end{equation}
    As for $\Cf_k (A)$ with large $k$, then one can easily obtain an analogues of formulae \eqref{f:Q_3_R}, \eqref{f:Q_4} but, generally speaking, the error term must be at least $2|A|^{k+1}$ 
%    (consider 
    (take 
    horizontal/vertical lines), and hence the consideration of large $k$ does not 
    give anything new.

    \bigskip

    In this paper we make a further step and obtain a paucity result  (see Theorems \ref{t:paucity}, \ref{t:f_A_5} below) for higher $\Cf_k (A)$.  
    As we said before, in general, one cannot obtain non--trivial bounds for such quantities and thus we need some restrictions, which we formulate in terms of {\it energies}.
    For any two sets $A,B \subseteq \F$ the {\it additive energy} of $A$ and $B$ is defined by
$$
    \E^{+} (A,B) = |\{ (a_1,a_2,b_1,b_2) \in A\times A \times B \times B ~:~ a_1 - b^{}_1 = a_2 - b^{}_2 \}| \,.
$$
If $A=B$, then  we simply write $\E^{+} (A)$ for $\E^{+} (A,A)$. Similarly, one can define the {\it multiplicative energy} of $A$ and $B$. Finally, put $\ov{\E}^\times (A) = \max_{s\in \F} \E^{\times} (A-s) \le |A|^3$.

%Having a function $f:\F \to \C$,  we put $\ov{\E}^\times (f) = \max_{s\in \F} \E^{\times} (f-s)$.
%In particular, $\ov{\E}^\times (A)\le |A|^3$ for any finite set $A$, $A \subseteq \F$.
    
\begin{theorem}
    Let $A,B\subset \mathbb{R}$ be sets.
    %, $|B| \le |A|$. 
    %, $|B| \le |A|$.
    Then 
    %for $k\ge 4$ one has 
\[
    \Cf_4 (A,B) - |B| |A|^4 - |A| |B|^4 
    \ll 
    |B| |A|^3 + |A| |B|^3 + 
\]
\begin{equation}\label{f:paucity_R_intr}
%    \ll
     (\E^{+} (A) \E^{+} (B))^{1/10} 
     \Cf^{2/5}_3 (A,B) |A|^{8/5} |B|^{6/5} \log^{2/5} |A| + 
    (\ov{\E}^{\times}(A) \ov{\E}^{\times}(B))^{1/4} |A|^2 |B|^{3/2} \log^{3/2} |A| \,.
\end{equation}
    Further let $p$ be a prime number, $A\subset \F_p$. Then 
\begin{equation}\label{f:f_A_5_Q_intr}
    \Cf_5 (A) - \frac{|A|^{10}}{p^3} - 2|A|^6 \lesssim  \frac{|A|^7}{p} +  \frac{|A|^4}{p^2} \Cf_3(A) 
%\]
%\begin{equation}\label{f:f_A_5_Q}
  +
    (\E^{+} (A))^{1/6} |A|^{11/2}  
    %\log^{5/6} |A|
        +
    (\ov{\E}^{\times}(A))^{1/2} |A|^{9/2}
    %\log^{3/2} |A| 
    \,.
\end{equation}
%\[
%    \Cf_5 (A) - \frac{|A|^{10}}{p^3} - 2|A|^6 \ll \frac{|A|^7 \log|A|}{p} +  \frac{|A|^4}{p^2} \Cf_3(A) + 
%\]
%\begin{equation}\label{f:f_A_5_Q_intr}
%  +
%    (\E^{+} (A))^{1/6} |A|^{11/2}  \log^{5/6} |A|
%        +
%    (\ov{\E}^{\times}(A))^{1/2} |A|^{9/2} \log^{3/2} |A| \,.
%\end{equation}
\label{t:paucity_intr} 
\end{theorem}

\begin{corollary}
    Let $A,B \subset \R$, $A=B$ and $\E^{+} (A), \ov{\E}^{\times}(A) \le |A|^{3-c}$, where $c>0$. Substituting estimate \eqref{f:Q_3_R} to Theorem \ref{t:paucity_intr}, we get 
\begin{equation}\label{tmp:CF_4_intr}
    \Cf_4 (A) - 2 |A|^5 \ll |A|^{5-c/5} \log^{4/5} |A| \,,
\end{equation}
and in the case of the prime field and, say, $|A| \le p^{2/3}$, we have 
\begin{equation}\label{tmp:CF_5_intr}
   \Cf_5 (A) - 2 |A|^6 \lesssim  |A|^{6-c/6} \log^{5/6} |A| \,.
\end{equation}
\end{corollary}

Once again the condition  $\E^{+}(A) \le |A|^{3-c}$ is required for \eqref{tmp:CF_4_intr}, \eqref{tmp:CF_5_intr}.
One can easily see that a random set $A$ satisfies both %requirements 
assumptions 
$\E^{+} (A), \ov{\E}^{\times}(A) \le |A|^{3-c}$, where $c>0$. Also, there are some concrete  constructions of such sets, see section \ref{sec:proof}, namely, one can  take $A=B^{-1}+s$ for any $s\in \F$ and $B\subseteq \F$ such that $|B+B|\ll |B|$ and then $\E^{+} (A), \ov{\E}^{\times}(A) \le |A|^{3-c}$. 
It seems like that our Theorem \ref{t:paucity_intr}  is the first %result  
paucity result in the area. 
%of such type. 
As for lower bounds on $\Cf_k (A)$ we trivially have $\Omega (|A|^k)$ in the error term, counting $k$--tuples, which belong to the diagonal $(x,x)$, where $x$ runs over $A$.

We obtain some applications of our paucity theorems in section \ref{sec:applications}. 
For example, let us formulate a consequence of our new incidence result in $\F_p$ concerning points/lines incidences. 
%Finally, let us remark that it is not possible to 
%It is easy to see that in general it is not possible to obtain  such formulae similar to \eqref{eq:incidences_new_intr} (basically, because the affine group group over the field is very far to being simple) but we obtain such a result under some additional restrictions on sets. 
The basic regime here is the following: the sets $A$, $B$ are large and the sets $X,Y$ are small comparable to $p$.

\begin{theorem}
    Let $k \ge 2$ be an integer, $M\ge 1$ be a real number,  $A,B,C,X,Y\subseteq \F_p$, $0\notin X$  be sets, $|YC|\le M|Y|$,
    and $|C| \gtrsim_k M^{2^{k+1}}$, $|C|^{k-1} \ge (p/|Y|)^2$. 
    Suppose that 
    $$
        \max\{ \E^{+} (B), \ov{\E}^\times (B) \} \le |B|^{3-\d} \,,
    $$
    where $\d>0$ is a constant. 
    Then 
    %the number of the solutions to the equation
\begin{equation}\label{eq:incidences_new_intr}
    |\{ (a,b,x,y)\in A\times B \times X\times Y ~:~ y = bx+a \}| 
        - \frac{|A||B||X||Y|}{p}
    \ll 
    \sqrt{|A||B|} |X| |Y| \cdot |B|^{-\frac{\d}{35 \cdot 2^k}} \,.
    %\,, 
\end{equation}
\label{t:incidences_new_intr}
\end{theorem}

There are two main advantages of Theorem \ref{t:incidences_new_intr}. First of all, it is an asymptotic formula (and, again, the set $X\times Y$, which corresponds to the lines $\L$ can be rather small, even of size $p^\eps$) but not just upper bounds for $\I (\mathcal{P}, \mathcal{L})$ as in \cite{SdZ}, say. 
Some asymptotic formulae for the quantity $\I (\mathcal{P}, \mathcal{L})$  were known before in the specific case of large sets (see \cite{Vinh})
%or estimate \eqref{f:Vinh} below) 
and in the case  of Cartesian products but with large sets of lines, see \cite{s_asymptotic} and  \cite{SdZ}.
Such asymptotic formulae are important in the problem of estimating exponential sums (see, e.g., \cite{s_asymptotic}) and in the questions on mixing times of some Markov chains \cite{s_He}, where one requires to have the set of points $A\times B$ to be really large. The second advantage of Theorem \ref{t:incidences_new_intr} is that it works even for $|A|, |B| \gg p$ (for size of the set $B$ see the exact formulation  of Theorem \ref{t:incidences_new} below).

\section{Definitions and preliminaries}

By $\Gr$ we denote an abelian group.
% with the identity $1$.
Sometimes we underline the group operation writing $+$ or $\times$ in  the considered quantities (as the energy, the representation function and so on, see below).  
Let $\F$ be the  field $\R$ or $\F=\F_p = \Z/p\Z$ for a prime $p$. Let $\F^* = \F \setminus \{0\}$.

We use the same capital letter to denote  set $A\subseteq \F$ and   its characteristic function $A: \F \to \{0,1 \}$ 
and in the case of finite $\F$ we write $f_A (x) := A(x) - |A|/|\F|$ for the {\it balanced function} of $A$.
Given two sets $A,B\subset \Gr$, define  
%the \textit{product set} (
the {\it sumset} 
%in the abelian case) 
of $A$ and $B$ as 
$$A+B:=\{a+b ~:~ a\in{A},\,b\in{B}\}\,.$$
In a similar way we define the {\it difference sets} and {\it higher sumsets}, e.g., $2A-A$ is $A+A-A$. 
We write $\dotplus$ for a direct sum, i.e., $|A\dotplus B| = |A| |B|$. 
For an abelian group $\Gr$
the Pl\"unnecke--Ruzsa inequality (see, e.g., \cite{TV}) holds stating
\begin{equation}\label{f:Pl-R} 
|nA-mA| \le \left( \frac{|A+A|}{|A|} \right)^{n+m} \cdot |A| \,,
\end{equation} 
where $n,m$ are any positive integers. 
%It follows from a more general inequality contained in  \cite{Petridis} that  for arbitrary sets $A,B,C \subseteq \Gr$ one has 
%\begin{equation}\label{f:Petridis}
%    |B+C+X| \le \frac{|B+X|}{|X|} \cdot  |C+X| \,,
%\end{equation}
%where $X\subseteq A$ minimizes  the quantity $|B+X|/|X|$.
We  use representation function notations like  $r_{A+B} (x)$ or $r_{A-B} (x)$ and so on, which counts the number of ways $x \in \Gr$ can be expressed as a sum $a+b$ or  $a-b$ with $a\in A$, $b\in B$, respectively. 
For example, $|A| = r_{A-A} (0)$.

For any two sets $A,B \subseteq \Gr$ the {\it additive energy} of $A$ and $B$ is defined by
$$
\E (A,B) = \E^{+} (A,B) = |\{ (a_1,a_2,b_1,b_2) \in A\times A \times B \times B ~:~ a_1 - b^{}_1 = a_2 - b^{}_2 \}| \,.
$$
If $A=B$, then  we simply write $\E^{} (A)$ for $\E^{} (A,A)$.
More generally, for sets (real functions) $A_1,\dots, A_{2k}$ ($f_1,\dots, f_{2k}$) belonging to an arbitrary (noncommutative) group $\Gr$ and 
%integer 
$k\ge 2$ define the energy 
$\T_{k} (A_1,\dots, A_{2k})$ as 
%как число решений уравнения 
\[
	\T_{k} (A_1,\dots, A_{2k}) 
	=
\]
\begin{equation}\label{def:T_k}
	=
	 |\{ (a_1, \dots, a_{2k}) \in A_1 \times \dots \times A_{2k} ~:~ a_1 a^{-1}_2  \dots a_{k-1} a^{-1}_k = a_{k+1} a^{-1}_{k+2}  \dots a_{2k-1} a^{-1}_{2k}  \}| \,,
\end{equation}
and 
\[
	\T_{k} (f_1,\dots, f_{2k}) 
	=
	\sum_{a_1 a^{-1}_2  \dots  a_{k-1} a^{-1}_k = a_{k+1} a^{-1}_{k+2} \dots a_{2k-1} a^{-1}_{2k}} f_1(a_1) \dots f_{2k} (a_{2k}) \,.
\]
In the abelian case put for $k\ge 2$  
\begin{equation}\label{def:E_k}
\E^{+}_k (A) = \sum_x r^k_{A-A} (x) = \sum_{\a_1, \dots, \a_{k-1}} |A\cap (A+\a_1) \cap \dots  \cap (A+\a_{k-1})|^2 \,.
\end{equation}
Clearly, $|A|^k \le \E^{+}_k (A) \le |A|^{k+1}$.
%Also, we write $\hat{\E}^{+}_k (A) = \sum_x r^k_{A+A} (x)$.  
Having a function $f:\F \to \C$,  we put $\ov{\E}^\times (f) = \max_{s\in \F} \E^{\times} (f-s)$.
%and if $A \subseteq \F$ is a set, then $\ov{\E}^\times (A) $
In particular, $\ov{\E}^\times (A)\le |A|^3$ for any finite set $A$, $A \subseteq \F$.

The signs $\ll$ and $\gg$ are the usual Vinogradov symbols.
When the constants in the signs  depend on a parameter $M$, we write $\ll_M$ and $\gg_M$. 
All logarithms are to base $2$.
If we have a set $A$, then we will write $a \lesssim b$ or $b \gtrsim a$ if $a = O(b \cdot \log^c |A|)$, $c>0$.
Let us denote by $[n]$ the set $\{1,2,\dots, n\}$.
%By $K_{s,t}$ denote the complete subgraph with two parts of sizes $s$ and $t$.  

\bigskip

Given an arbitrary line $l$
%, intersecting $A\times B$ in at least two points, 
denote by $i_{A,B} (l) = |l\cap (A\times B)|$ hence in particular, $\mathcal{I} (A\times B,\mathcal{L})=\sum_{l\in L} i_{A,B} (l)$.
Clearly, for $k\ge 3$ one has 
\begin{equation}\label{def:Q_k_i} 
    \Cf_k(A,B) = \sum_{l~:~i_{A,B} (l)>1} i^k_{A,B} (l) + O_k ( |A|^2 |B|^2 \cdot \left( \min\{|A|, |B|\} \right)^{k-3} \,.
\end{equation}
%where $|\theta| \le 1$. 
One can think about the first term in \eqref{def:Q_k_i} as another definition of the quantity $\Cf_k(A,B)$.

%\bigskip 

%Notice that we need the term depending on $\E^{+} (A)$, $\E^{+} (B)$  in  Theorem \ref{t:paucity} to obtain an asymptotic formula. 
%Indeed, if $A=B = [n]$, $n\ll \sqrt{|\F|}$, then one can show that  $\Cf_3 (A) \gg |A|^4 \log^c |A|$, where $c>0$ is an absolute constant, see details in \cite[Appendix A, page 633]{collinear}. 

\bigskip

We need a result on the energy of an arbitrary set of affine transformations, see \cite[Theorem 2.1]{PR-NRW}.
We consider any set of affine transformations $\mathcal{L} \subseteq \Aff(\F)$ as a subset of $\F^* \times \F$, i.e., we associate a map $x\to ax+b$ with a point $(a,b) \in \F^* \times \F$.

\begin{theorem}
    Let $\mathcal{L} \subseteq \Aff(\F)$ be a set of lines  such that no line in $\F^* \times \F$ contains more than $M(\mathcal{L})$ points of $\mathcal{L}$, and no vertical line in $\F^* \times \F$  contains more than $m(\mathcal{L})$ points of $\mathcal{L}$. 
    Suppose that $m(\mathcal{L})|\mathcal{L}| \le p^2$.
    Then 
$$
    \E(\mathcal{L}) \ll m^{1/2}(\mathcal{L}) |\mathcal{L}|^{5/2} + M(\mathcal{L}) |\mathcal{L}|^2 \,.
$$
\label{t:Aff_energy}
\end{theorem}

The following result is contained in \cite[Proposition 7]{s_He}. 
Here we have considered the case of the prime field (which is more difficult) and the case of the real numbers can be treated in analogues way. 

\begin{proposition}
    Let $A,B \subseteq \F$ be sets and $\mathcal{L}$ be a set of affine transformations over $\F$. 
    Then for any positive integer $k$ one has 
\begin{equation}\label{f:counting_collinear}
    \I (A\times B, \mathcal{L}) - \frac{|A||B||\mathcal{L}|}{|\F|} 
    \ll 
    \sqrt{|A||B| |\mathcal{L}|} \cdot 
    (\T_{2^k} (\mathcal{L}) |A| \log |A|)^{1/2^{k+2}} \,. 
\end{equation}
    More precisely, if $\F = \R$, then
\begin{equation}\label{f:counting_collinear_R}
    \I (A\times B, \mathcal{L}) \ll 
    \sqrt{|A||B|} |\mathcal{L}|^{1/3} \cdot 
    (\T_{2^k} (\mathcal{L}) |A| \log |A|)^{1/3\cdot 2^{k}} \,. 
\end{equation}
\label{p:counting_collinear}
\end{proposition}

We need the following structural result on the higher energies, see \cite[Theorem 4]{s_Sidon}. 

\begin{theorem}
	Let $\Gr$ be an abelian group, $A\subseteq \Gr$ be a set, $\delta, \eps \in (0,1]$ be parameters, $\eps \le \delta$.\\
	$1)~$ Then there is $k=k(\d, \eps) = \exp(O(\eps^{-1} \log (1/\d)))$ such that either $\E^{}_k (A) \le |A|^{k+\delta}$ or there is $H\subseteq \Gr$, $|H| \gtrsim  |A|^{\delta(1-\eps)}$, $|H+H| \ll |A|^\eps |H|$ 
	and there exists $Z\subseteq \Gr$,  $|Z| |H| \ll |A|^{1+\eps}$ 
	%such that 
	with 
	$$|(H\dotplus Z) \cap A| \gg |A|^{1-\eps} \,.$$ 
	$2)~$ Similarly, either there is a set $A'\subseteq A$, $|A'| \gg |A|^{1-\eps}$ and $P\subseteq \Gr$, $|P| \gtrsim |A|^\d$ such that for all $x\in A'$ one has  $r_{A-P}(x) \gg |P| |A|^{-\eps}$ or $\E_k (A) \le |A|^{k+\d}$ with $k \ll 1/\eps$.
\label{t:Ek} 
\end{theorem}

Thus either $\E_k(A) \le |A|^{k+\d}$, or $A$ has some strong structural properties depending on the parameters $\d$ and $\eps$. In the later case we say that $A$ is {\it $\E_k$--exceptional with $\d,\eps$.}

\bigskip 

We use a simplified version of 
%this result 
\cite[Theorem 5]{s_Bourgain}, as well as \cite[Lemma 9]{s_He} 
in our last section.

\begin{theorem}
    Let  $k\ge 2$ be an integer, $M\ge 1$ be a real number,
    $A,B\subseteq \F_p$ be sets, $|AB|\le M|A|$, and $|B| \gtrsim_k M^{2^{k+1}}$. 
    Then 
\begin{equation}\label{f:upper_T(AB)}
    \T^{+}_{2^k} (A) \lesssim_k M^{2^{k+1}} \left( \frac{|A|^{2^{k+1}}}{p}  +
    |A|^{2^{k+1}-1} \cdot |B|^{-\frac{k-1}{2}}  \right) \,.  
\end{equation}
\label{t:upper_T(AB)+}
\end{theorem}

\begin{lemma}
    Let $A,B \subseteq \F^*_p$ be sets, %$K=|AA|/|A|$ 
    and $\mathcal{L} = \{ (a,b) ~:~ a\in A,\, b\in B\} \subseteq \Aff(\F_p)$. 
    Then for any $k\ge 2$ one has 
\begin{equation}\label{f:T_k-E(L)}
    \T_k (f_{\mathcal{L}}) \le  |A|^{2k-1} \T^{+}_k (f_B) \,.
\end{equation}
\label{l:T_k-E(L)}
\end{lemma}

\section{The proof of the main result}
\label{sec:proof}

First of all, combining Theorem \ref{t:Aff_energy} and Proposition \ref{p:counting_collinear}, we obtain  a result on "rich"\, lines (i.e., lines having large intersections with $A\times B$) in terms of the quantities $m(\mathcal{L})$ and $M(\mathcal{L})$. 
We follow the method from \cite{RS_SL2} and \cite{PR-NRW}. 
Given a set $\mathcal{L}$ of affine transformations over $\F$ %$\mathcal{L}$  
and a real number $\tau \ge 2$ put 
\[
    \mathcal{L}_\tau = \L_\tau (A) := \{ l\in \mathcal{L} ~:~ |l\cap (A\times B)| \ge \tau \} \,.
\]
Formulae \eqref{f:S_t_Aff_1}, \eqref{f:S_t_Aff_2} below can be treated as an alternative (to the Szemer\'edi--Trotter type estimates) upper bound for size of $|\mathcal{L}_\tau|$.

\begin{corollary}
    Let $A,B \subseteq \F$ be sets and $\mathcal{L}$ be a set of affine transformations over $\F$. 
    Also, let $\tau\ge 1$ be a real number, $\frac{2|A||B|}{|\F|} \le \tau$ and $m(\mathcal{L}_\tau) |\mathcal{L}_\tau| \le |\F|^2$.
    %\le m:=\min \{|A|,|B|\}$. 
    Then 
\begin{equation}\label{f:S_t_Aff_1}
    |\mathcal{L}_\tau| \ll \tau^{-16/3} m^{1/3} (\mathcal{L}_\tau) |A|^{10/3} |B|^{8/3} \log^{2/3} |A| 
    + \tau^{-4} M^{1/2} (\mathcal{L}_\tau) |A|^{5/2} |B|^2 \log^{1/2} |A| \,.
\end{equation}
%    provided $m|\mathcal{L}_\tau| \le |\F|^2$ and $|\mathcal{L}_\tau| \ge m$. 
    More precisely, if $\F = \R$, then
\begin{equation}\label{f:S_t_Aff_2}
    |\mathcal{L}_\tau| \ll \tau^{-4} (m (\mathcal{L}_\tau) |A|^8 |B|^6 \log^2 |A|)^{1/3} + \tau^{-3} M^{1/2} (\mathcal{L}_\tau) |A|^2 |B|^{3/2} \log^{1/2} |A| \,. 
\end{equation}
\label{c:S_t_Aff}
\end{corollary}
\begin{proof} 
    Let $m=m(\mathcal{L}_\tau)$ and $M= M(\mathcal{L}_\tau)$. 
    By Proposition \ref{p:counting_collinear} with $k=1$, the definition of the set $\mathcal{L}_\tau$ and our condition $\frac{2|A||B|}{|\F|} \le \tau$, we get 
\[
    \tau |\mathcal{L}_\tau| \ll \sqrt{|A| |B| |\mathcal{L}_\tau|} \cdot 
    (\E (\mathcal{L}_\tau) |A| \log |A|)^{1/8} \,.
\]
%    Thanks to our assumptions $m|\mathcal{L}_\tau| \le |\F|^2$, $|\mathcal{L}_\tau| \ge m$  
    We now apply Theorem \ref{t:Aff_energy} and obtain 
\begin{equation}\label{f:L_tau_m,M_E}
    \tau |\mathcal{L}_\tau| \ll \sqrt{|A| |B| |\mathcal{L}_\tau|} (|A| \log |A|)^{1/8}
    \cdot 
    (m^{1/2} |\mathcal{L}_\tau|^{5/2} + M |\mathcal{L}_\tau|^{2})^{1/8} \,.
\end{equation}
    The last inequality is equivalent to \eqref{f:S_t_Aff_1}.
    To obtain estimate \eqref{f:S_t_Aff_2} we just use formula  \eqref{f:counting_collinear_R} instead of \eqref{f:counting_collinear}.
%as required. 
    This completes the proof. 
$\hfill\Box$
\end{proof}

\bigskip 

%We estimate 
Now let us give simple upper bounds for 
the quantities  $m(\mathcal{L}_\tau)$, $M(\mathcal{L}_\tau)$.

\begin{lemma}
    Let $A,B \subseteq \F$ be sets, $\tau \ge 2$ and $\mathcal{L}$ be a set of affine transformations over $\F$. 
    Then for any integer $k\ge 2$ one has 
    $$
        m(\mathcal{L}_\tau) \le \tau^{-k} (\E^{+}_k (A) \E^{+}_k (B))^{1/2} 
        \quad \quad 
            \mbox{ and }
        \quad \quad 
        M(\mathcal{L}_\tau) \le \tau^{-k}
        (\ov{\E}^{\times}_k (A) \ov{\E}^{\times}_k (B))^{1/2} \,.
    $$
\label{l:m,M_bounds}
\end{lemma}
\begin{proof} 
    Let $m=m(\mathcal{L}_\tau)$ and $M=M(\mathcal{L}_\tau)$. 
    By the definition of the quantity $m$ we can find $m$ lines in $\mathcal{L}_\tau$ having the form $\a x +\beta$, where $\a\neq 0$ is a fixed number and $\beta$ runs over a set $\mathcal{B}$ of cardinality $m$. 
    Using the definition of the set $\mathcal{L}_\tau$ and 
    %the Cauchy--Schwarz inequality, 
    the H\"older inequality, 
    we obtain 
\[
    \tau m \le \sum_{\beta \in \mathcal{B}}\, \sum_{x\in B} A(\a x+\beta) = \sum_{\beta \in \mathcal{B}} r_{\a B-A} (\beta) 
    \le m^{1-1/k} (\E^{+}_k (\a B,A))^{1/k} \le m^{1-1/k} (\E^{+}_k (A) \E^{+}_k (B) )^{1/2k} \,.
\]
    Similarly, by the definition of the quantity $M$ we find $M$ lines concurrent at a certain  point $(x_0,y_0)$.
    We parameterise them by their slopes: $\beta = \frac{y-y_0}{x-x_0}$, $\beta \in \mathcal{B}$, $|\mathcal{B}| = M$. Hence as above
\[
    \tau M \le \sum_{\beta \in \mathcal{B}}\, r_{(A-y_0)/(B-x_0)} (\beta) \le  M^{1-1/k} (\E^{\times} (A-y_0,B-x_0))^{1/2} 
    \le M^{1-1/k} (\ov{\E}^{\times}_k (A) \ov{\E}^{\times}_k (B))^{1/2k} \,.
\]
%as required. 
    This completes the proof. 
$\hfill\Box$
\end{proof}

\bigskip 

Now we are ready to obtain our main result in the case of the real field and a result on $\Cf_5 (A,B)$ for small $A$ and $B$ in the case of the prime field.

\begin{theorem}
    Let $A,B\subset \mathbb{R}$ be sets.
    %, $|B| \le |A|$. 
    %, $|B| \le |A|$.
    Then 
    %for $k\ge 4$ one has 
\[
    \Cf_4 (A,B) - |B| |A|^4 - |A| |B|^4 
    \ll 
    |B| |A|^3 + |A| |B|^3 + 
\]
\begin{equation}\label{f:paucity_R}
%    \ll
     (\E^{+} (A) \E^{+} (B))^{1/10} 
     \Cf^{2/5}_3 (A,B) |A|^{8/5} |B|^{6/5} \log^{2/5} |A| + 
    (\ov{\E}^{\times}(A) \ov{\E}^{\times}(B))^{1/4} |A|^2 |B|^{3/2} \log^{3/2} |A| \,.
\end{equation}
    If $A,B \subseteq \F_p$ such that $\Cf^2_4 (A,B) \log |A| \le p |A|^5 |B|^4$ and $|A|^5 |B|^6 \le p^9$, then 
\[
    \Cf_5 (A,B) - |B| |A|^5 - |A| |B|^5 
    \ll 
    |B| |A|^4 + |A| |B|^4 + \Cf'_5 (A,B) \,,
\]
where the error term $\Cf'_5 (A,B)$  is at most 
\begin{equation}\label{f:paucity}
    (\E^{+} (A) \E^{+} (B))^{1/12} \Cf^{1/2}_4 (A,B) |A|^{5/3} |B|^{4/3} \log^{1/3} |A|
        +
    (\ov{\E}^{\times}(A) \ov{\E}^{\times}(B))^{1/4} |A|^{5/2} |B|^2 \log^{3/2} |A| \,.
\end{equation}
\label{t:paucity} 
\end{theorem}
\begin{proof} 
    We start with \eqref{f:paucity_R}. 
    The term $|B| |A|^4 + |A| |B|^4$ corresponds to vertical/horizontal lines plus an error, which is
    %the rest is 
    $O(|B| |A|^3 + |A| |B|^3)$. 
    Let $\mathcal{L}$ be the set of non--vertical and non--horizontal lines, having at least two points in $A\times B$. 
    Denote by $\mathcal{E}$ 
    %the right--hand side of 
    the error term in 
    \eqref{f:paucity}. 
    Let $\D>0$ be a parameter, which we will choose later. 
    %Applying estimate \eqref{f:Q_3_R} 
    Having the definition of the quantity $\Cf_3 (A,B)$
    and using the diadic Dirichlet principle, we get 
\[
    \mathcal{E} \ll \D^{} 
    %|A|^2 |B|^2 \log |A| 
    \Cf_3 (A,B) 
    + \sum_{j~:~2^j \ge \D} 2^{4j} |\mathcal{L}_{2^j}|  = 
    \D^{} \Cf_3 (A,B)
    %|A|^2 |B|^2 \log |A| 
    + \mathcal{E}' \,.
\]
    Put $m_* = (\E^{+} (A) \E^{+} (B))^{1/2}$, $M_* = (\ov{\E}^{\times}(A) \ov{\E}^{\times}(B))^{1/2}$.
    %Using 
    Applying 
    bound \eqref{f:S_t_Aff_2} of Corollary \ref{c:S_t_Aff}, combining with Lemma \ref{l:m,M_bounds} (with the parameter $k=2$),  we obtain 
\[
    \mathcal{E}' \ll \sum_{j~:~2^j \ge \D} 2^{4j} \cdot \left(
    2^{-14j/3} (m_* |A|^8 |B|^6 \log^2 |A|)^{1/3} + 2^{-4j} M^{1/2}_* |A|^2 |B|^{3/2} \log^{1/2} |A| \right)
\]
\[
    \ll
    \D^{-2/3} (m_* |A|^8 |B|^6 \log^2 |A|)^{1/3}
    +
    M^{1/2}_* |A|^2 |B|^{3/2} \log^{3/2} |A| \,.
\]
    Now choosing $\D^{5/3} = (m_* |A|^8 |B|^6 \log^2 |A|)^{1/3} \Cf^{-1}_3 (A,B)$ (one can check that $\D\ge 2$ thanks to formula \eqref{def:Q_k_i}, say), we get 
\[
    \mathcal{E} \ll \Cf^{2/5}_3 (A,B) m^{1/5}_* |A|^{8/5} |B|^{6/5} \log^{2/5} |A| + 
    M^{1/2}_* |A|^2 |B|^{3/2} \log^{3/2} |A| 
    %\,.
\]
as required. 
To obtain \eqref{f:paucity} we use the same argument, namely, 
\begin{equation}\label{tmp:05.10_1}
\mathcal{E} \ll \D^{} 
    %|A|^2 |B|^2 \log |A| 
    \Cf_4 (A,B) 
    + \sum_{j~:~2^j \ge \D} 2^{5j} |\mathcal{L}_{2^j}|  = 
    \D^{} \Cf_4 (A,B)
    %|A|^2 |B|^2 \log |A| 
    + \mathcal{E}' \,,
\end{equation}
and 
\begin{equation}\label{tmp:05.10_2}
    \mathcal{E}' \ll \sum_{j~:~2^j \ge \D} 2^{5j}
    \left( 
    2^{-6j} m^{1/3}_* |A|^{10/3} |B|^{8/3} \log^{2/3} |A| 
    + 2^{-5j} M^{1/2}_* |A|^{5/2} |B|^2 \log^{1/2} |A|
    \right) 
\end{equation}
\begin{equation}\label{tmp:05.10_3}
    \ll
    \D^{-1} m^{1/3}_* |A|^{10/3} |B|^{8/3} \log^{2/3} |A| 
    +
    M^{1/2}_* |A|^{5/2} |B|^2 \log^{3/2} |A| \,.
\end{equation}
    Now the optimal choice of $\D$ is $\D^2 = m^{1/3}_* |A|^{10/3} |B|^{8/3} \log^{2/3} |A| \cdot \Cf^{-1}_4 (A,B)$ (again $\D\ge 2$ thanks to formula  \eqref{def:Q_k_i}, say) and hence 
\[
    \mathcal{E} \ll 
        \Cf^{1/2}_4 (A,B) m^{1/6}_* |A|^{5/3} |B|^{4/3} \log^{1/3} |A|
        +
    M^{1/2}_* |A|^{5/2} |B|^2 \log^{3/2} |A| \,.
\]
    It remains to check that $\D \ge 2 |A||B|/p$ and $m(\mathcal{L}_\Delta) |\mathcal{L}_\Delta| \le p^2$.
    %Since $\Cf_4 (A,B) \ge |A|^4 |B| + |A| |B|^4$, it follows that $\Cf_4 (A,B) \gg (|A||B|)^{5/2}$ 
    Since by our condition\\ $\Cf^2_4 (A,B) \log |A| \le p |A|^5 |B|^4$, it follows that 
    %and hence 
    in view of trivial lower bounds $\E^{+}(A) \ge |A|^2$, $\E^{+}(B) \ge |B|^2$, we get 
\[
    \D^2 \gg m^{1/3}_* |A|^{5/6} |B|^{2/3} p^{-1/2} \ge |A|^{7/6} |B|^{} p^{-1/2} \,.
\]
    It is easy to see that the last quantity is greater than $(2 |A||B|/p)^2$ thanks to $|A|^5 |B|^6 \le p^9$. 
    Finally, using $|\mathcal{L}_\D| \le \Cf_4 (A,B) \D^{-4}$, Lemma \ref{l:m,M_bounds} with $k=2$, the definition of the quantity $\D$ and our assumption, we obtain 
\begin{equation}\label{tmp:mL_check}
   m(\mathcal{L}_\Delta) |\mathcal{L}_\Delta| \le  
   m_* \Cf_4 (A,B) \D^{-6} \le \Cf^4_4 (A,B) |A|^{-10} |B|^{-8} \log^2 |A| \le p^2 \,.
\end{equation}
%as required. 
    This completes the proof. 
$\hfill\Box$
\end{proof}

\bigskip

Of course one can rewrite the condition $\Cf^2_4 (A,B) \log |A| \le p |A|^5 |B|^4$ in terms of cardinalities of $A$ and $B$, using  \eqref{f:Q_4} but we leave this a little bit more precise condition. %$\Cf^2_4 (A,B) \log |A| \le p |A|^5 |B|^4$
Now let us consider a constructive family of sets satisfying the assumptions of Theorem \ref{t:paucity}.

\bigskip

{\bf Example.} Let $A\subseteq \F$, $|A| <\sqrt{|\F|}$, say, $|A+A| \le K|A|$ and consider $X:=A^{-1}$. Then one can quickly  show that $\E^{+}(X) \ll_K |A|^{3-1/4}$ (see \cite[Lemma 14]{s_asymptotic}) 
and $\E^{\times} (X+s) \ll_K |A|^{3-c}$ for a certain $c>0$ and any $s \in \F$.
We give a sketch of the proof of the last estimate. 
Indeed, for $s=0$ it immediately follows from Rudnev's Theorem \cite{Rudnev_pp} (with $c=1/2$) and for $s\neq 0$ (dividing we can suppose that $s=1)$ 
%the equation for $\E^{\times} (X+s)$ is 
we see that it is enough to solve the equation 
\begin{equation}\label{f:E(X)_Warren}
    \left(\frac{1}{a+b} +1 \right) \left(\frac{1}{c+d} +1 \right) =\la \,,
\end{equation}
    where $\la\neq 1$ is a fixed number. 
    If we consider $a,c$ as variables and $b,d$ as coefficients, then, clearly, \eqref{f:E(X)_Warren} determines a family of conics  and the number of the solutions to the equation can be estimated via the main result of  \cite{MPW}, say.

\bigskip

Similarly, we obtain an analogue of estimate \eqref{f:paucity} without any restrictions on size of $A$. 

\begin{theorem}
    Let $A\subseteq \F_p$ be  a set and $f_A (x) = A(x)-|A|/p$. Then 
\begin{equation}\label{f:f_A_5}
%\[
    \sum_{l \in \Aff(\F_p)} \left| \sum_x f_A(x) f_A(lx) \right|^5 \lesssim 
%\]
%    \ll
    (\E^{+} (f_A))^{1/6} |A|^{11/2}  
        +
    (\ov{\E}^{\times}(f_A))^{1/2} |A|^{9/2}  \,,
\end{equation}
    where the summation over $l$ in the last formula is taken over all affine transformations   
    having at least two points in $A\times A$. 
    In particular,
\begin{equation}\label{f:f_A_5_Q}
    \Cf_5 (A) - \frac{|A|^{10}}{p^3} - 2|A|^6 \lesssim  \frac{|A|^7}{p} +  \frac{|A|^4}{p^2} \Cf_3(f_A) 
%\]
%\begin{equation}\label{f:f_A_5_Q}
  +
    (\E^{+} (f_A))^{1/6} |A|^{11/2}  
    %\log^{5/6} |A|
        +
    (\ov{\E}^{\times}(f_A))^{1/2} |A|^{9/2}
    %\log^{3/2} |A| 
    \,.
\end{equation}
\label{t:f_A_5}
\end{theorem} 
\begin{proof} 
    Let $\mathcal{E}$ be the left--hand side of \eqref{f:f_A_5}. 
    First of all, we redefine the set $\mathcal{L}_\tau$ as 
\[
    \mathcal{L}_\tau = \L_\tau (f_A) := \left\{ l\in \mathcal{L} ~:~ \left| \sum_x f_A(x) f_A(lx) \right| \ge \tau \right\} \,.
\]
    Then estimate \eqref{f:counting_collinear} of Proposition \ref{p:counting_collinear}  with $A=B$ gives us 
\begin{equation}\label{tmp:05.11_1}
    \tau |\mathcal{L}_\tau| \ll \sqrt{|A| |B| |\mathcal{L}_\tau|} \cdot 
    (\E (\mathcal{L}_\tau) |A| \log |A|)^{1/8} \,.
\end{equation}
    Indeed, 
\[
    \tau |\mathcal{L}_\tau| \le \sum_{l\in \mathcal{L}_\tau} \left| \sum_x f_A(x) f_A(lx) \right| = 
    \sum_{l\in \mathcal{L}_\tau} \eps(l) \sum_x f_A(x) f_A(lx) \,,
\]
    where by $\eps(l)$ we have denoted the sign of $\left| \sum_x f_A(x) f_A(lx) \right|$. 
    In other words, we 
    consider 
    %have obtained 
    a new function $\mathcal{L}^\eps_\tau (l) := \mathcal{L}_\tau (l) \eps(l)$. 
    Clearly, $\|\mathcal{L}^\eps_\tau \|_1 = |\mathcal{L}_\tau|$ and for  any integer $l\ge 2$ one has 
    $\T_{l} (\mathcal{L}^\eps_\tau) \le \T_{l} (\mathcal{L}_\tau)$.
    Thus using the Cauchy--Schwarz inequality as in the proof of \cite[Proposition 7]{s_He} (alternatively  see the proof of Theorem  \ref{t:incidences_new} below), we obtain \eqref{tmp:05.11_1} (without any condition on $\tau$). 
    Thus Proposition \ref{p:counting_collinear} (estimate \eqref{tmp:05.11_1}) holds for such defined $\mathcal{L}_\tau$ and hence Corollary \ref{c:S_t_Aff} takes place as well. 
    As for an analogue of Lemma \ref{l:m,M_bounds}, we have for $m=m(\mathcal{L}_\tau)$ and an arbitrary even $k$ that 
    %and the signs $\eps (\beta)$, $\beta \in \mathcal{B}$ that 
\[
    \tau m \le  \sum_{\beta \in \mathcal{B}}\, \left| \sum_{x\in B} f_A (\a x+\beta) \right| 
%    \sum_{\beta \in \mathcal{B}}\, \eps(\beta) \sum_{x\in B} f_A (\a x+\beta) 
    = 
    \sum_{\beta \in \mathcal{B}}  |r_{\a B-f_A} (\beta)| 
%    \le 
%\]
%\[
    \le 
    m^{1-1/k} (\E^{+}_k (\a f_B,f_A))^{1/k} 
    \le
\]
\[
    \le 
    m^{1/2} (\E^{+}_k (f_A) \E^{+}_k (f_B))^{1/4} \,.
\]
    and similarly for $M(\mathcal{L}_\tau)$. 
    Finally, in \cite{collinear} it was proved that 
\begin{equation}\label{f:collinear-}
    \sigma = \sum_{l \in \Aff(\F_p)} \left| \sum_x f_A(x) f_A(lx) \right|^4 \ll |A|^5 \log |A| \,, 
\end{equation}
    where again  the summation over $l$ in the last formula is taken over all affine transformations having at least two points in $A\times A$. Actually, \eqref{f:collinear-} is equivalent to asymptotic formula \eqref{f:Q_4}.
    Thus we can repeat the calculations in \eqref{tmp:05.10_1}---\eqref{tmp:05.10_3} and obtain 
\[
    \mathcal{E} \ll \sigma^{1/2} m^{1/6}_* |A|^{3}  \log^{1/3} |A|
        +
    M^{1/2}_* |A|^{9/2} \log^{3/2} |A| 
    \ll 
\]
\[
    \ll
    m^{1/6}_* |A|^{11/2}  \log^{5/6} |A|
        +
    M^{1/2}_* |A|^{9/2} \log^{3/2} |A| 
    \,,
\]
    where $m_* = \E^{+} (f_A)$, $M_* = \ov{\E}^{\times}(f_A)$.
%as required. 
    It remains to check the condition $m(\LL_\tau)|\LL_\tau| \le p^2$ as in \eqref{tmp:mL_check}. 
    Using these calculations, as well as \eqref{f:collinear-}, we see that the condition $|A|\ll p/\log^3 |A|$ is enough. 
    Splitting our set $A$ if its needed and loosing some logarithms, we arrive to \eqref{f:f_A_5} for all $A$.

    Now let us obtain \eqref{f:f_A_5_Q}. 
    The term $2|A|^6$ corresponds to vertical/horizontal lines and since $\E^{+} (f_A)), \ov{\E}^{\times}(f_A) \gg |A|^2$, it follows that all appearing terms, which less than $|A|^{35/6}$ are negligible. 
    Let $i(l) = |l\cap (A\times A)|$ and below we consider just $l$ with $i(l)>1$. Then $\sum_l i^2 (l) = |A|^4 -|A|^2$ and 
\[
    \sum_{l} i^5(l) = \sum_l i^2(l) \left( i(l) - \frac{|A|^2}{p} + \frac{|A|^2}{p} \right)^3
    =
    \sum_l i^2(l) \left( i(l) - \frac{|A|^2}{p} \right)^3 
    +
\]
\[
    +
    \frac{3|A|^2}{p} \sum_l i^2(l) \left( i(l) - \frac{|A|^2}{p} \right)^2 + \frac{3|A|^4}{p^2} \left( \sum_l i^3(l) - \frac{|A|^2}{p} \sum_l i^2(l) \right) + \frac{|A|^{10}}{p^3} - \frac{|A|^8}{p^3} \,. 
\]
    It is easy to check that 
%    We have 
    %in view of \eqref{f:Q_3} that 
\[
    \sum_l i^3(l) - \frac{|A|^2}{p} \sum_l i^2(l)  \le  \Cf_3(f_A) \,.
    %+ \frac{|A|^{4}}{p} 
\]
%    The terms of the form $O(|A|^8/p^3)$ are negligible.
    %less than $|A|^5$ and hence 
    The sum $\sum_l i^2(l) \left( i(l) - \frac{|A|^2}{p} \right)^2$ was estimated in \cite[page 606]{collinear}, namely, splitting the summation over $i(l)<2|A|^2/p$ and $i(l) \ge 2|A|^2/p$, as well as using \eqref{f:Q_4}, we get 
\[
    \sum_l i^2(l) \left( i(l) - \frac{|A|^2}{p} \right)^2
    \ll
        \left(\frac{|A|^2}{p} \right)^2 p|A|^2 + |A|^5 \log |A| 
        \ll |A|^5 \log |A|
\]
    and hence we have the remaining term $O(p^{-1} |A|^7 \log |A|)$.
    Finally,
\[
    \sum_l \left( i(l) -\frac{|A|^2}{p} + \frac{|A|^2}{p}  \right)^2 \left( i(l) - \frac{|A|^2}{p} \right)^3 = \sum_l \left( i(l) - \frac{|A|^2}{p} \right)^5 + \frac{2|A|^2}{p} \sum_l \left( i(l) - \frac{|A|^2}{p} \right)^4 
    +
\]
\[
    + \frac{|A|^4}{p^2} \sum_l \left( i(l) - \frac{|A|^2}{p} \right)^3
    = \mathcal{E} + O( p^{-1} |A|^7 \log |A| + \frac{|A|^4}{p^2} \Cf_3(f_A)) \,.
\]
    Combining all bounds and estimate \eqref{f:f_A_5}, we obtain 
    \eqref{f:f_A_5_Q}. 
    This completes the proof. 
$\hfill\Box$
\end{proof}

\section{Some applications}
\label{sec:applications}

The power saving in \eqref{f:f_A_5} (for sets with small additive/multiplicative energies of its shifts) allows us to obtain our new incidence result Theorem \ref{t:incidences_new_intr} from the introduction. 
%Of course, the same is true for the equation $$
We follow the method of the proof from  \cite[Proposition 7]{s_He}.

\begin{theorem}
    Let $k\ge 2$ be an integer, $M\ge 1$ be a real number,   $A,B,C,X,Y\subseteq \F_p$, $0\notin X$ be sets, $|YC|\le M|Y|$,
    and $|C| \gtrsim_k M^{2^{k+1}}$, $|C|^{k-1} \ge (p/|Y|)^2$. 
    Suppose that $\max\{ \E^{+} (f_B), \ov{\E}^\times (f_B) \} \le |B|^{3-\d}$, where $\d>0$ is a constant. 
    Then 
    %the number of the solutions to the equation
\begin{equation}\label{eq:incidences_new}
    |\{ (a,b,x,y)\in A\times B \times X\times Y ~:~ y = bx+a \}| 
        - \frac{|A||B||X||Y|}{p}
    \ll 
    \sqrt{|A||B|} |X| |Y| \cdot |B|^{-\frac{\d}{35 \cdot 2^k}} \,.
    %\,, 
\end{equation}
%    where $\eps (\d)>0$. 
\label{t:incidences_new}
\end{theorem}
\begin{proof} 
    At the beginning we repeat the arguments from the proofs of  \cite[Proposition 8]{s_He} and Proposition \ref{p:counting_collinear}. 
    Let $\LL$ be 
    %an arbitrary  
    the 
    set of the lines $\{ l_{x,y}\}$, $x\in X$, $y\in Y$ 
    %with coefficients $x,y$ 
     defined as $y=bx+a$, so $a,b$ are variables.
    Thus the equation from the left--hand side of \eqref{eq:incidences_new} can be treated as a question about incidences between  points $\P= A\times B$ and lines $\LL$.
    We have 
    \begin{equation}\label{f:I_balanced} 
    \I (A\times B, \mathcal{L}) = \frac{|A||B||\mathcal{L}|}{p}  +  \sum_{x\in A} \sum_{l \in \LL} f_B (l x) =
 \frac{|A||B||\mathcal{L}|}{p} + \sigma \,.   
\end{equation}
    Using the H\"older inequality several times, we get
\begin{equation}\label{tmp:04.10_1*}
    \sigma^{2^{k}} \le |A|^{2^{k-1}}  |B|^{2^{k-1}-1} 
    \sum_{h} r_{(\mathcal{L}^{-1} \mathcal{L})^{2^{k-1}}} (h) \sum_x f_B (x) f_B (h x) \,.
\end{equation}
Assume that the summation in the last formula is taken over lines $h$, having at least two points in $B\times B$, denote the rest as $\sigma_*$ and suppose that $\sigma_* \le 2^{-1} \sigma^{}$, say.
    Then applying the H\"older inequality one more time, combining with formula \eqref{f:f_A_5} of Theorem \ref{t:f_A_5}, we obtain 
\[
    \sigma^{5\cdot 2^{k}} \le |A|^{5\cdot 2^{k-1}}  |B|^{5\cdot 2^{k-1}-5} \T_{2^k} (\LL) |\LL|^{3\cdot 2^k}
    \cdot 
    \sum_h \left| \sum_x f_B (x) f_B (h x) \right|^5 
    \ll
\]
\[
    \ll 
    |A|^{5\cdot 2^{k-1}}  |B|^{5\cdot 2^{k-1}-5} \T_{2^k} (\LL) |\LL|^{3\cdot 2^k}  |B|^{6-\d/7} \,.
\]
    By Lemma \ref{l:T_k-E(L)}, we know that $\T_{2^k} (\LL) \le |X|^{2^{k+1}-1} \T^{+}_{2^k} (Y)$. 
    Using our conditions $|C| \gtrsim_k M^{2^{k+1}}$, $|C|^{k-1} \ge (p/|Y|)^2$, as well as Theorem \ref{t:upper_T(AB)+} for $A=Y$ and $B=C$ to estimate the quantity $\T^{+}_{2^k} (Y)$, we obtain 
\begin{equation}\label{tmp:29_12_1}
    \sigma \ll \sqrt{|A| |B|} |X||Y| (|B|^{1-\d/7} /(|X|p))^{1/5\cdot 2^k}
    \le 
        \sqrt{|A| |B|} |X||Y| \cdot |B|^{-\d/(35 \cdot 2^k)} \,.
\end{equation}
    It remains to estimate $\sigma_*$. 
    Returning to 
    %\eqref{f:I_balanced}, 
    \eqref{tmp:04.10_1*}, we see that
\[
    \sum_x f_B (x) f_B (h x) = \sum_x B (x) B (h x) - \frac{|B|^2}{p}
\]
    and hence 
$$
    \sigma^{5\cdot 2^{k}}_* \ll |A|^{5\cdot 2^{k-1}}  |B|^{5\cdot 2^{k-1}-5} |\mathcal{L}|^{5\cdot 2^{k}} \,.
%    \cdot \max\left\{ \frac{|Y|^{10}}{p^5}, 1 \right\} \,.
$$
 %     If the maximum is attained at one, then we obtain a better result than in \eqref{tmp:29_12_1}. 
    and this is better  than  \eqref{tmp:29_12_1}. 
%      $|Y|^8/p^5$, then $\sigma_*$ is negligible because, clearly, $\T_{2^k} (\LL) \ge |\LL|^{2^{k+1}}/p$ and $|A|\le p$.
    This completes the proof. 
$\hfill\Box$
\end{proof}

%\section{}

\bigskip

We now obtain another application, using our structural Theorem \ref{t:Ek} (also, see the discussion and the definitions after this result). We show that for all sets having no special form one can obtain a good power saving for $\Cf_n (f_A)$. It allows us to estimate nontrivially (in a rather strong sense) some mixed energies of shifts of $A$, see Theorem \ref{t:E_4} below.

\begin{theorem} 
    Let $A\subseteq \F$ be  a set and $f_A (x) = A(x)-|A|/|\F|$.
    Suppose that $A$ is not $\E^{+}$--exceptional and for any $s\in \F$ the set $A-s$  is not $\E^{\times}$--exceptional with $\d,\eps$. 
    %$\max\{\E^{+}(f_A), \overline{\E}^\times (f_A)\} \le |A|^{3-c}$. 
    Then for any  number $n\ge \exp(C\eps^{-1}\log(1/\d))$, where $C>0$ is an absolute constant one has 
\begin{equation}\label{f:f_A_k}
%\[
    \sum_{l \in \Aff(\F)} \left| \sum_x f_A(x) f_A(lx) \right|^n \ll |A|^{n+2/3+\d} \,.
\end{equation}
%    In particular, if $n\ge $, then 
%\begin{equation}\label{f:f_A_k+}
%    \sum_{l \in \Aff(\F_p)} \left| \sum_x f_A(x) f_A(lx) \right|^n \ll |A|^n \,.
%\end{equation}
\label{t:f_A_k}
\end{theorem}
\begin{proof} 
    Suppose that $n\ge 6$ is  an even number and 
    let $\mathcal{E}_n$ be the left--hand side of \eqref{f:f_A_k}, i.e. the $n$th moment of the function $\sum_x f_A(x) f_A(lx)$. 
%    Our aim is to prove that $\E_n \le |A|^{n+1-c(n-4)/7}$ {\bf ???} for $n\ge 5$ and for $n=5$ is follows from   Theorem \ref{t:f_A_5}. 
    We mimic  the calculations in \eqref{tmp:05.10_1}---\eqref{tmp:05.10_3} and obtain
    %via the $(n-1)$th moment that 
\begin{equation}\label{f:E_n_basic}
%    \mathcal{E}_n \ll \D \mathcal{E}_{n-1} + \sum_{j: 2^j \ge \D} 2^{nj} |\L_{2^j}| \,.
    \mathcal{E}_n \ll \sum_{j} 2^{nj} |\L_{2^j}| \,.
\end{equation}
    Now to estimate $|\L_{2^j}|$ we use 
    %inequality \eqref{f:L_tau_m,M_E}, 
    formula \eqref{f:S_t_Aff_1} of Corollary \ref{c:S_t_Aff} and 
    Lemma \ref{l:m,M_bounds} with $k_1=3n-16\ge 2$, $k_2 = 2n-8 \ge 4$ (hence $k_1$, $k_2$ are  even numbers automatically) 
    %and the following trivial bound $\E^{+}_k (f_A) \le |A|^{k-2} \E^{+} (f_A) \le |A|^{k+1-c}$ and similarly, $\overline{\E}^\times_k (f_A) \le |A|^{k+1-c}$
    to get 
\[
    |\mathcal{L}_\tau| \ll \tau^{-(k_1+16)/3} (\E^{+}_{k_1} (f_A))^{1/3} |A|^6 \log^{2/3} |A| 
    + \tau^{-(k_2+8)/2} (\ov{\E}^{\times}_{k_2} (f_A))^{1/2} |A|^{9/2} \log^{1/2} |A| \,.
\]
    Hence summing the last estimate over $j$ and using  our basic 
    %estimate 
    bound 
    \eqref{f:E_n_basic}, we get 
\[
    \mathcal{E}_n \ll \sum_{j} 2^{nj-j(k_1+16)/3} (\E^{+}_{k_1} (f_A))^{1/3} |A|^6 \log^{2/3} |A| 
    + \sum_{j} 2^{nj-j(k_2+8)/2} (\ov{\E}^{\times}_{k_2} (f_A))^{1/2} |A|^{9/2} \log^{1/2} |A| 
%    \ll 
\]
\[
    \ll 
    (\E^{+}_{3n-16} (f_A))^{1/3} |A|^6 \log^{5/3} |A| + (\ov{\E}^{\times}_{2n-8} (f_A))^{1/2} |A|^{9/2} \log^{3/2} |A| 
    \,.
\]
    Now applying Theorem \ref{t:Ek} and the condition that $A-s$ is not exceptional for both energies $\E^{+}$, $\E^\times$ with $\d,\eps$, we derive 
\[
    \mathcal{E}_n \ll |A|^{n+2/3+\d/3} \log^{5/3} |A|
    \ll |A|^{n+2/3+\d/10} 
\] 
	as required. 
%    This completes the proof. 
$\hfill\Box$
\end{proof}

\bigskip 

Of course one can obtain an analogue of asymptotic formula \eqref{f:f_A_5_Q} of Theorem \ref{t:f_A_5} for larger $k$ but it requires some calculations and we leave it for the interested reader.

%\section{Applications} 

\bigskip 

Having sets $A,B\subset \F$ and an integer $k\ge 2$
let $\tilde{\Cf}_k(A,B)$ be $\Cf_k (A,B) - |A||B|^k - |B| |A|^k$, that is, we have deleted all horizontal/vertical lines from our consideration. Also, we put 
\[
    \tilde{\E}^\times_k (A,B) 
    = 
    \left| \left\{ (a_1,\dots,a_k, b_1,\dots, b_k) \in (A\setminus\{0\})^k \times (B\setminus\{0\})^k ~:~ \frac{a_1}{b_1} = \dots = \frac{a_k}{b_k} \right\} \right| \,.
\]
%%For simplicity suppose that $0\notin A,B$. 
Then from the equation of a line $y=\lambda x+ \mu$, intersecting $A\times B$, we derive
\begin{equation}\label{f:t_Ck}
    \tilde{\Cf}_k(A,B) 
    + O_k ( |A|^2 |B|^2 \cdot \left( \min\{|A|, |B|\} \right)^{k-3}) 
    = \sum_{\lambda \neq 0} \E^{+}_k (B,\lambda A) = \sum_{\mu} \tilde{\E}^\times_k (B-\mu,A) 
    %\,.
\end{equation}
similar to 
%thanks to 
%equation 
formula 
\eqref{def:Q_k_i}. 
As will we see from the proof of Theorem \ref{t:E_4} there are other expressions for $\Cf_k(A,B)$ in terms of higher energies $\E_k$. 
Notice that in the symmetric case $A=B$ we always have $\tilde{\E}^\times_k (A,A) \gg |A|^k$.
%$\E^{\times}_k (A) \ge |A|^k$ but the quantity $\tilde{\E}^\times_k (A,A)$ can be smaller. 
In contrast, using our paucity result, we obtain

\begin{theorem}
    Let $A,B\subset \F_p$ be sets, $|A|=|B| \le p^{2/3}$, and $\E^{+} (A)$, $\ov{\E}^{\times}(A) \le |A|^{3-c}$, where $c\in [0,1)$. 
    Then there are $b_1,b_2\in B$ such that 
\begin{equation}\label{f:E_4_1}
    \tilde{\E}^\times_4 (A-b_1,A-b_2) \ll |A|^{4-\frac{2c}{15}} \log |A| \,.
\end{equation}
    Now suppose that our set $A$ is not $\E^{+}$--exceptional and for any $s\in \F$ the set $A-s$  is not $\E^{\times}$--exceptional with $\d,\eps$. 
    Then for any  number $n\ge \exp(C\eps^{-1}\log(1/\d))$, where $C>0$ is an absolute constant 
    %one has for 
    one can find 
    a pair $b_1,b_2\in B$ with  
\begin{equation}\label{f:E_4_2}
    \tilde{\E}^{\times}_n (f_A-b_1,f_A-b_2) \ll |A|^{n-1/3+\delta} \,. 
\end{equation}
\label{t:E_4}
\end{theorem}
\begin{proof} 
    Let $L = \log |A|$. 
    The number of collinear quintuplets $(b_1,b_2), (a_1,a'_1), \dots, (a_4,a'_4) \in (B\times B) \times (A\times A)^4$ equals the number $\sigma$ of the solutions to the system 
\[
    \frac{a_1-b_1}{a'_1 - b_2} = \dots = \frac{a_4-b_1}{a'_4 - b_2}
\]
    with non--zero numerators and denominators plus quintuplets, which correspond to vertical/ho\-ri\-zon\-tal lines. 
    Thus the sum 
    $\sigma := \sum_{b_1,b_2\in B} \tilde{\E}^\times_4 (A-b_1,A-b_2)$ can be estimated via formula \eqref{f:paucity} of Theorem \ref{t:paucity}.
%    as 
%    $$ 
%        \sum_{b_1,b_2\in B} \tilde{\E}^\times_4 (A-b_1,A-b_2)
%            \ll 
%            \sum_{b_2\in B} \tilde{\Cf}_4 (A,A-b_2) + |A|^6 
%    $$
    We cannot write $\sigma \le \t{\Cf}_5(A,A)^{4/5} \t{\Cf}_5 (B,B)^{1/5}$ using the H\"older inequality because one can have $i_{A,A} (l) = 1$, say.
    %or $i_{B,B} (l) = 1$. 
    Nevertheless, the contribution of these  terms in view of formula \eqref{f:Q_4} is at most 
    %$|A|^2 |B|^3 + 
    $\Cf_4 (A\cup B) \ll |A|^5 L$ and hence it is negligible. 
    Thus by estimate \eqref{f:paucity} and our restrictions $|A|=|B| \le p^{2/3}$, we obtain 
\[
    \sum_{b_1,b_2\in B} \tilde{\E}^\times_4 (A-b_1,A-b_2) \ll (|A|^{6-c/6} L^{5/6} + |A|^5 L)^{4/5} (|A|^6 L + |A|^5 L)^{1/5}
    + |A|^5 L
    \ll |A|^{6-2c/15} L 
    %\,.
\]
    and we have proved bound \eqref{f:E_4_1}. 
    To obtain \eqref{f:E_4_2} we just apply Theorem \ref{t:f_A_k} instead of Theorem \ref{t:paucity}. 
This completes the proof. 
$\hfill\Box$
\end{proof}

\bigskip 

Of course, in a similar way, using formula  \eqref{f:t_Ck} and our bounds for $\tilde{\Cf}_k(A,B)$, one can derive some upper bounds for 
$\E^{+}_k (B,\lambda A)$,  $\tilde{\E}^\times_k (B-\mu,A)$ %with 
for typical shifts 
$\lambda, \mu$.

\bigskip

\noindent{I.D.~Shkredov\\
Steklov Mathematical Institute,\\
ul. Gubkina, 8, Moscow, Russia, 119991}
\\
%and
%\\
%IITP RAS,  \\
%Bolshoy Karetny per. 19, Moscow, Russia, 127994\\
%and 
%\\
%MIPT, \\ 
%Institutskii per. 9, Dolgoprudnii, Russia, 141701\\
{\tt ilya.shkredov@gmail.com}

\end{document}